\documentclass[a4paper,11pt]{amsart}
\usepackage[utf8x]{inputenc}
\usepackage{amsthm,amsfonts,amsmath,amssymb,enumerate}
\usepackage{verbatim} 
\usepackage{rotating} 

\usepackage{hyperref}
\usepackage{pdfsync}
\usepackage{multirow}

\renewcommand{\L}{\mathcal{L}}

\newcommand{\C}{\mathbb{C}}
\newcommand{\g}{\mathfrak{g}}

\newcommand{\h}{\mathfrak{h}}

\renewcommand{\deg}{\operatorname{deg}}
\renewcommand{\O}{\mathcal{O}}

\newtheorem{theorem}{Theorem}[section]

\theoremstyle{definition}

\theoremstyle{remark}
\newtheorem{remark}[theorem]{Remark}
\numberwithin{equation}{section}

\begin{document}

\title[Filiforms as degenerations]{Filiform Lie algebras of dimension 8 \\ as degenerations}
\author{Joan Felipe Herrera-Granada and  Paulo Tirao}
\address{CIEM-FaMAF, Universidad Nacional de C\'ordoba, Argentina}
\date{July 29, 2013}

\maketitle

\begin{abstract}
For each complex 8-dimensional filiform Lie algebra we find another non isomorphic Lie algebra that degenerates to it.
Since this is already known for nilpotent Lie algebras of rank $\ge 1$, 
only the caracteristically nilpotent ones should be considered.
\end{abstract}

\section{Introduction}

In this short paper we show that each complex 8-dimensional filiform Lie algebra is the degeneration of another non isomorphic Lie algebra.
This adds more evidence supporting the Grunewald-O'Halloran conjecture,
which states that every nilpotent Lie algebra is the degeneration of another non isomorphic Lie algebra.
This conjecture is stronger than Vergne's conjecture which states that there are no rigid nilpotent Lie algebras in the variety of all Lie algebras.

In the previous paper \cite{HGT} we proved that all nilpotent Lie algebras of rank $\ge 1$ are the degeneration of another
non isomorphic Lie algebra, living open the Grunewald-O'Halloran conjecture for caracteristically nilpotent Lie algebras.
In \cite{HGT} we considered also all 7-dimensional caracteristically nilpotent Lie algebras proving 
the conjeture for all 7-dimensional complex nilpotent Lie algebras.

We now consider the filiform Lie algebras of dimension 8, which have been classified in \cite{ABG}.
According to this classification, there are four families parametrized by a complex parameter 
and six isolate filiform algebras all of which are caracteristically nilpotent.
For each of these we construct, following \cite{GO}, a non trivial linear deformation that corresponds to a degeneration.
We note that it is not known which ones of the linear deformations constructed in \cite{GO} does correspond to a degeneration and which ones does not. 
We exhibit an example of such a deformation which does not correspond to a degeneration.

The case of filiform Lie algebras is of particular interest.
Degeneration is transitive and under degeneration the nilpotency degree does not grow.
Hence, the nilpotent Lie algebras of maximal nilpotency class, the filiforms,
are very high on the diagram of degenerations.
Filiforms may degenerate to any other nilpotent Lie algebra, but only a filiform may degenerate to a given filiform.

This paper is very much related to our paper \cite{HGT}.
We keep definitions and notations from it and we include here only some basic preliminaries, to make this paper more easy to read.

\section{Preliminaries}

Let $\L_n$ be the algebraic variety of complex Lie algebras of dimension $n$,
together with the action of the group $GL_n=GL_n(\C)$ by `change of basis',
and denote the orbit of $\mu$ in $\L_n$ by $\O(\mu)$. 

A Lie algebra $\mu$ degenerates to a Lie algebra $\lambda$, $\mu \rightarrow_{\deg} \lambda$, if $\lambda\in\overline{\O(\mu)}$,
the Zariski closure of $\O(\mu)$. 

A linear deformation of a Lie algebra $\mu$ is, for the aim of this paper, a family $\mu_t$, $t\in \C^\times$,
of Lie algebras such that
\[ \mu_t=\mu + t\phi, \]
where $\phi$ is a Lie algebra bracket which in addition is a 2-cocycle of $\mu$.

If a linear deformation $\mu_t$ of $\mu$ is such that $\mu_t\in\O(\mu_1)$ for all $t\in\C^\times$, 
then $\mu_1\rightarrow_{\deg}\mu$.
In fact, for each $t\in\C^\times$ there exist $g_t\in GL_n$ such that $g_t^{-1}\cdot \mu_1=\mu_t$,
then $\lim_{t\mapsto 0}g_t^{-1}\cdot \mu_1=\lim_{t\mapsto 0}\mu_t=\mu$.
Then, in order to show that $\mu_1\rightarrow_{\deg}\mu$, we only need to prove that 
for each $t\in\C^\times$ there exist $g_t\in GL_n$ such that
\begin{equation}\label{eqn:degeneration}
 \mu_1(g_t(x),g_t(y)))=g_t(\mu_t(x,y)), \quad\text{for all $x,y\in\C^n$}.
\end{equation}

Let $(\g,\mu)$ be a given Lie algebra of dimension $n$ and let $\h$ be a codimension 1 ideal of $\g$ with a semisimple derivation $D$.
For any element $X$ of $\g$ outside $\h$, $\g=\langle X \rangle \oplus \h$.
The bilinear form $\mu_D$ on $\g$ defined by $\mu_D(X,z)= D(z)$ and $\mu_D(y,z)=0$, for $y,z\in\h$, is a 2-cocycle for $\mu$
and a Lie bracket.
Hence,
\begin{equation}\label{eqn:linear-deformation}
 \mu_t=\mu + t\mu_D,
\end{equation}
is a linear deformation of $\mu$.
If $\g$ is nilpotent, then $\mu_t$ is always solvable but not nilpotent. 
In particular, $\mu_t$ is not isomorphic to $\mu$ for all $t\in\C^\times$.
The construction described was given in \cite{GO}.

\section{8-dimensional filiforms}

Complex filiform Lie algebras of dimension 8 have been classified in \cite{ABG}.
This classification is presented as a list of 6 families index by a parameter $\alpha\in\C$
\[ 
 \mu_8^{6,\alpha}, \quad \mu_8^{7,\alpha}, \quad \mu_8^{9,\alpha}, \quad \mu_8^{10,\alpha}, \quad
 \mu_8^{13,\alpha}, \quad \mu_8^{14,\alpha},
\]
and 14 isolated algebras
\[
\begin{gathered}
 \mu_8^{1}, \quad \mu_8^{2}, \quad \mu_8^{3}, \quad \mu_8^{4}, \quad \mu_8^{5}, \quad \mu_8^{8}, \quad \mu_8^{11}, \\
 \mu_8^{12}, \quad \mu_8^{15}, \quad \mu_8^{16}, \quad \mu_8^{17}, \quad \mu_8^{18}, \quad \mu_8^{19}, \quad \mu_8^{20}.
\end{gathered}
\]

The following table shows those algebras with a semisimple derivation. 
We keep the name and the bases from \cite{ABG}.

\begin{center}
\begin{tabular}{|c|c|c|c|}\hline
\scriptsize{$\mu$} & \scriptsize{$D\in Der(\mu)$} & \scriptsize{$\mu$} & \scriptsize{$D\in Der(\mu)$}\\ \hline 
\rule[-1.5cm]{0cm}{3.2cm}\scriptsize{$\mu_{8}^{3}$} & \scriptsize{$\left(\begin{smallmatrix}
4 &  &  &  &  &  &  &  \\
 & 13 &  &  &  &  &  &  \\
 &  & 9 &  &  & 3 &  &  \\
 &  &  & 8 &  &  & 6 &  \\
 &  &  &  & 7 &  &  & 6 \\
 &  &  &  &  & 6 &  &  \\
 &  &  &  &  &  & 5 &  \\
3 &  &  &  &  &  &  & 1 \\
\end{smallmatrix}\right)$} & \scriptsize{$\mu_{8}^{14,\alpha}$} & \scriptsize{$\left(\begin{smallmatrix}
1 &  &  &  &  &  &  &   \\
 & 9 &  &  &  &  &  &   \\
 &  & 8 &  &  &  &  &   \\
 &  &  & 7 &  &  &  &   \\
 &  &  &  & 6 &  &  &   \\
 &  &  &  &  & 5 &  &   \\
 &  &  &  &  &  & 4 &   \\
 &  &  &  &  &  &  & 3  \\
\end{smallmatrix}\right)$}\\ \hline 
\rule[-1.6cm]{0cm}{3.4cm}\scriptsize{$\mu_{8}^{4}$} & \scriptsize{$\left(\begin{smallmatrix}
3 &  &  &  &  &  &  &   \\
 & 11 &  & -2 &  &  &  &   \\
 &  & 8 &  & -2 &  &  &   \\
 &  &  & 7 &  &  &  &   \\
 &  &  &  & 6 &  & 2 &   \\
 &  &  &  &  & 5 &  & 2  \\
 &  &  &  &  &  & 4 &   \\
2 &  &  &  &  &  &  & 1  \\
\end{smallmatrix}\right)$} & \scriptsize{$\mu_{8}^{16}$} & \scriptsize{$\left(\begin{smallmatrix}
3 &  &  &  &  &  &  &   \\
 & 27 &  &  &  &  &  &   \\
 &  & 24 &  & 2 &  & 2 &   \\
 &  &  & 21 &  & 4 &  & 2  \\
 &  &  &  & 18 &  & 6 &   \\
 &  &  &  &  & 15 &  & 6  \\
 &  &  &  &  &  & 12 &   \\
2 &  &  &  &  &  &  & 9  \\
\end{smallmatrix}\right)$}\\ \hline
\rule[-1.3cm]{0cm}{2.8cm}\scriptsize{$\mu_{8}^{5}$} & \scriptsize{$\left(\begin{smallmatrix}
1 &  &  &  &  &  &  &   \\
 & 7 &  &  &  &  &  &   \\
 &  & 6 &  &  &  &  &   \\
 &  &  & 5 &  &  &  &   \\
 &  &  &  & 4 &  &  &   \\
 &  &  &  &  & 3 &  &   \\
 &  &  &  &  &  & 2 &   \\
 &  &  &  &  &  &  & 1  \\
\end{smallmatrix}\right)$} & \scriptsize{$\mu_{8}^{18}$} & \scriptsize{$\left(\begin{smallmatrix}
1 &  &  &  &  &  &  &   \\
 & 10 &  &  &  &  &  &   \\
 &  & 9 &  &  &  &  &   \\
 &  &  & 8 &  &  &  &   \\
 &  &  &  & 7 &  &  &   \\
 &  &  &  &  & 6 &  &   \\
 &  &  &  &  &  & 5 &   \\
 &  &  &  &  &  &  & 4  \\
\end{smallmatrix}\right)$}\\ \hline
\rule[-1.7cm]{0cm}{3.6cm}\scriptsize{$\mu_{8}^{7,\alpha}$} & \scriptsize{$\left(\begin{smallmatrix}
1 &  &  &  &  &  &  &   \\
 & 8 &  &  &  &  &  &   \\
 &  & 7 &  &  &  &  &   \\
 &  &  & 6 &  &  &  &   \\
 &  &  &  & 5 &  &  &   \\
 &  &  &  &  & 4 &  &   \\
 &  &  &  &  &  & 3 &   \\
 &  &  &  &  &  &  & 2  \\
\end{smallmatrix}\right)$} & \scriptsize{$\mu_{8}^{19}$} & \scriptsize{$\left(\begin{smallmatrix}
1 &  &  &  &  &  &  &   \\
 & 11 &  &  &  &  &  &   \\
 &  & 10 &  &  &  &  &   \\
 &  &  & 9 &  &  &  &   \\
 &  &  &  & 8 &  &  &   \\
 &  &  &  &  & 7 &  &   \\
 &  &  &  &  &  & 6 &   \\
 &  &  &  &  &  &  & 5  \\
\end{smallmatrix}\right)$}\\ \hline
\rule[-1.5cm]{0cm}{3.2cm}\scriptsize{$\mu_{8}^{12}$} & \scriptsize{$\left(\begin{smallmatrix}
1 &  &  &  &  &  &  &   \\
 & 8 &  &  &  &  &  &   \\
 &  & 7 &  &  &  &  &   \\
 &  &  & 6 &  &  &  &   \\
 &  &  &  & 5 &  &  &   \\
 &  &  &  &  & 4 &  &   \\
 &  &  &  &  &  & 3 &   \\
 &  &  &  &  &  &  & 2  \\
\end{smallmatrix}\right)$} & \scriptsize{$\mu_{8}^{20}$} & \scriptsize{$\left(\begin{smallmatrix}
1 &  &  &  &  &  &  &   \\
 & 0 &  &  &  &  &  &   \\
 &  & -1 &  &  &  &  &   \\
 &  &  & -2 &  &  &  &   \\
 &  &  &  & -3 &  &  &   \\
 &  &  &  &  & -4 &  &   \\
 &  &  &  &  &  & -5 &   \\
 &  &  &  &  &  &  & -6  \\
\end{smallmatrix}\right)$}\\ \hline
\end{tabular}
\end{center}

The remaining algebras do not have any semisimple derivation, they are all caracteristically nilpotent, that is, their derivation algebras are nilpotent as Lie algebras.
These are
\begin{equation}\label{eqn:caracteristically-nilpotent}
 \begin{gathered}
 \mu_8^{6,\alpha}, \quad \mu_8^{9,\alpha}, \quad \mu_8^{10,\alpha}, \quad \mu_8^{13,\alpha}, \\
 \mu_8^{1}, \quad \mu_8^{2}, \quad \mu_8^{8}, \quad \mu_8^{11}, \quad \mu_8^{15}, \quad \mu_8^{17}.
\end{gathered}
\end{equation}

\section{Degenerations to filiforms}

We now show that for every complex filiform Lie algebra of dimension 8 there is another Lie algebra non isomorphic to it,
actually solvable non nilpotent, that degenerates to it.
That solvable algebra is constructed as a linear deformation of the original.
This is the main result of this paper.

\begin{theorem}
 Every 8-dimensional complex filiform Lie algebras is the degeneration of a solvable Lie algebra.
\end{theorem}

\begin{proof}
The result was already stablished in \cite{HGT} for nilpotent Lie algebras of rank $\ge 1$, that is, admiting a semisimple derivation.
Hence, we only need to consider the caracteristically nilpotent ones, listed in \eqref{eqn:caracteristically-nilpotent}.

We proceed on a case by case basis.
For each one of the characteristically nilpotent filiforms $\mu$, we choose an ideal $\h$ of codimension 1 togheter with a semisimple derivation $D$ of it,
and consider the corresponding linear deformation $\mu_t$ given by \eqref{eqn:linear-deformation}.
We then show that it corresponds to a degeneration,
by showing a family of linear isomorphisms $g_t\in GL(8,\C)$ satisfying \eqref{eqn:degeneration}.
The tables below, one for each algebra or family of algebras considered, contain the choosen ideal, the semisimple derivation,
and the family $g_t$ realizing the degeneration.

For the ease of presentation in all cases we choose a new basis $\{Y_1,\dots,Y_8\}$ of $\C^8$, that we present in terms of
the basis $\{X_1,\dots,X_8\}$ used in \cite{ABG}.
We denote by $\mu_i$ and by $\mu_{i,\alpha}$ respectively the brackets $\mu_8^i$ and $\mu_8^{i,\alpha}$ in this new basis.
In all cases, the ideal of codimension 1 we choose is $\h=\langle Y_2,\dots,Y_8 \rangle$.
The derivation $D$ choosen and the family of linear isomorphisms $g_t$ are written in the latest basis.
With strightforward computations one can verify that $g_t$ satisfies \eqref{eqn:degeneration}.

\begin{center} \tiny
 \begin{tabular}{|c|c|} \hline
   ${\{Y_1,\dots,Y_8\}}$ & $\mu_{1}$ \\ \hline
   
   $\begin{aligned}\ \newline\\  Y_1 &= X_1 \\ Y_2 &= \tfrac{2}{5}X_6-X_7+X_8 \\ Y_3 &= \tfrac{51}{5000}X_3+\tfrac{1}{250}X_4+\tfrac{17}{50}X_5-X_6+X_7 \\ 
        Y_4 &= -\tfrac{44}{125}X_3+\tfrac{77}{50}X_4-\tfrac{44}{5}X_5+11X_6 \\ Y_5 &= -\tfrac{33}{500}X_3+\tfrac{33}{50}X_4-\tfrac{11}{10}X_5 \\ 
        Y_6 &= -\tfrac{11}{50}X_4 \\ Y_7 &= \tfrac{11}{1000}X_3 \\ Y_8 &= X_2  \\ \ \newline
     \end{aligned}$
   & $\begin{gathered}\ 
       [Y_1,Y_2] = Y_3-\tfrac{3}{55}Y_5-\tfrac{8}{55}Y_6-\tfrac{69}{55}Y_7, \\ 
       [Y_1,Y_3] = \tfrac{1}{11}Y_4+\tfrac{2}{11}Y_5-\tfrac{4}{11}Y_6+\tfrac{48}{11}Y_7+\tfrac{51}{5000}Y_8, \\ 
       [Y_1,Y_4] = -10Y_5+10Y_6+80Y_7-\tfrac{44}{125}Y_8 \\
       [Y_1,Y_5] = 5Y_6+60Y_7-\tfrac{33}{500}Y_8,\quad [Y_1,Y_6] = -20Y_7, \\
       [Y_1,Y_7] = \tfrac{11}{1000}Y_8, \quad [Y_2,Y_3] = -\tfrac{1}{11}Y_4, \\
       [Y_2,Y_4] = 10Y_5, \quad [Y_2,Y_5] = -5Y_6, \\
       [Y_2,Y_6] = 20Y_7, \quad [Y_3,Y_4] = 400Y_7, \\
       [Y_3,Y_6] = \tfrac{11}{50}Y_8, \quad [Y_4,Y_5] = -\tfrac{121}{10}Y_8
     \end{gathered}$ \\ \hline
     
     \tiny{$D=\left(\begin{smallmatrix}
1 &   &   &   &   &   &   \\
  & 3 &   &   &   &   &   \\
  &   & 4 &   &   &   &   \\
  &   &   & 5 &   &   &   \\
  &   &   &   & 6 &   &   \\
  &   &   &   &   & 7 &   \\
  &   &   &   &   &   & 9 \\
\end{smallmatrix}\right)$}  &  \rule[-1.5cm]{0cm}{3.2cm}\tiny{$g_{t}=\left(\begin{smallmatrix} 
t &  &  &  &  &  &  &  \\[0.1cm]
0 & t &  &  &  &  &  &  \\[0.1cm]
p_{1} & p_{2} & t^{3} &  &  &  &  &  \\[0.1cm]
p_{3} & 0 & 0 & t^{4} &  &  &  &  \\[0.1cm]
p_{4} & 0 & p_{5} & 0 & t^{5} &  &  &  \\[0.1cm]
p_{6} & 0 & p_{7} & p_{8} & 0 & t^{6} &  &  \\[0.1cm]
0 & p_{9} & p_{10} & p_{11} & p_{12} & 0 & t^{7} &  \\[0.1cm]
0 & 0 & p_{13} & p_{14} & p_{15} & 0 & p_{16} & t^{9}\\
\end{smallmatrix}\right)$}  \\ \hline
   \multicolumn{2}{|l|}{$\begin{gathered}\ \newline\\
    p_{1}=-\tfrac{1}{2}t(t-1),\quad p_{2}=\tfrac{1}{2}t(t-1), \quad p_{3}=\tfrac{1}{550}t(3t^{3}-5t^{2}+10t-8), \\
    p_{4}=-\tfrac{2}{825}t(12t^{4}+10t^{3}-25t+3), \quad p_{5}=\tfrac{1}{11}t^{3}(t-1), \\ 
    p_{6}=\tfrac{1}{1100}t(39t^{5}-45t^{4}-15t^{3}-10t^{2}+65t-49), \quad p_{7}=-\tfrac{4}{33}t^{3}(t^{2}-1),\\
    p_{8}=5t^{4}(t-1),\quad p_{9}=-\tfrac{1}{22}t(2t^{5}-t^{4}+4t^{3}-16t^{2}+32t-20), \quad 
    p_{10}=\tfrac{1}{11}t^{3}(18t^{3}+5t^{2}-10t-13),\\
    p_{11}=\tfrac{80}{3}t^{4}(t^{2}-1),\quad p_{12}=30t^{5}(t-1), \quad 
    p_{13}=\tfrac{1}{2000}t^{3}(6t^{5}+7t^{4}-17t^{3}+t^{2}+t+1), \\
    p_{14}=\tfrac{11}{75}t^{4}(t^{3}-t^{2}-t+1),\quad p_{15}=\tfrac{11}{200}t^{5}(t^{2}-2t+1), \quad
    p_{16}=\tfrac{11}{2000}t^{7}(t-1)\\ \ \newline
    \end{gathered}$} \\ \hline
 \end{tabular}
\end{center}

\begin{center} \tiny
 \begin{tabular}{|c|c|} \hline
   ${\{Y_1,\dots,Y_8\}}$ & $\mu_{2}$ \\ \hline
   
   $\begin{aligned}\ \newline\\  Y_1 &= X_1 \\ Y_2 &= \tfrac{1}{2}X_3-\tfrac{3}{8}X_5+X_8 \\
    Y_3 &= X_2-\tfrac{1}{2}X_4+\tfrac{1}{8}X_6 \\ 
    Y_4 &= -\tfrac{3}{2}X_3+\tfrac{3}{4}X_5 \\ Y_5 &= \tfrac{1}{2}X_4-\tfrac{1}{4}X_6 \\ 
    Y_6 &= -\tfrac{3}{8}X_5 \\ Y_7 &= \tfrac{1}{8}X_6 \\ Y_8 &= X_7  \\ \ \newline
     \end{aligned}$
   & $\begin{gathered}\ 
       [Y_1,Y_2] = -Y_3-2Y_5-4Y_7-\tfrac{1}{2}Y_8, \\ 
       [Y_1,Y_3] = \tfrac{2}{3}Y_4+\tfrac{8}{3}Y_6-8Y_7, \\ 
       [Y_1,Y_4] = 3Y_5+12Y_7+\tfrac{3}{2}Y_8, \\
       [Y_1,Y_5] = \tfrac{4}{3}Y_6, \quad [Y_1,Y_6] = 3Y_7, \\
       [Y_2,Y_3] = -\tfrac{2}{3}Y_4, \quad [Y_2,Y_4] = -3Y_5, \\
       [Y_2,Y_5] = -\tfrac{4}{3}Y_6, \quad [Y_2,Y_6] = -3Y_7, \\
       [Y_2,Y_7] = \tfrac{1}{8}Y_8, \quad [Y_3,Y_6] = \tfrac{3}{8}Y_8, \\
       [Y_4,Y_5] = -\tfrac{3}{4}Y_8
     \end{gathered}$ \\ \hline
     
     \tiny{$D=\left(\begin{smallmatrix}
1 &   &   &   &   &   &   \\
  & 2 &   &   &   &   &   \\
  &   & 3 &   &   &   &   \\
  &   &   & 4 &   &   &   \\
  &   &   &   & 5 &   &   \\
  &   &   &   &   & 6 &   \\
  &   &   &   &   &   & 7 \\
\end{smallmatrix}\right)$}  &  \rule[-1.5cm]{0cm}{3.2cm}\tiny{$g_{t}=\left(\begin{smallmatrix} 
t &  &  &  &  &  &  &  \\[0.1cm]
0 & t &  &  &  &  &  &  \\[0.1cm]
0 & 0 & t^{2} &  &  &  &  &  \\[0.1cm]
p_{1} & 0 & 0 & t^{3} &  &  &  &  \\[0.1cm]
p_{2} & 0 & 0 & 0 & t^{4} &  &  &  \\[0.1cm]
0 & 0 & p_{3} & 0 & 0 & t^{5} &  &  \\[0.1cm]
p_{4} & p_{5} & p_{6} & p_{7} & 0 & 0 & t^{6} &  \\[0.1cm]
0 & 0 & p_{8} & p_{9} & p_{10} & 0 & 0 & t^{7}\\
\end{smallmatrix}\right)$}  \\ \hline
   \multicolumn{2}{|l|}{$\begin{gathered}\ \newline\\
    p_{1}=-\tfrac{2}{3}t(t^{2}-1),\quad p_{2}=-\tfrac{2}{3}t(t^{2}-1), \quad p_{3}=\tfrac{8}{9}t^{2}(t^{2}-1), \\
    p_{4}=\tfrac{2}{15}t(30t^{5}-20t^{4}+3t^{3}+20t^{2}-33),\quad p_{5}=-\tfrac{2}{5}t(2t^{4}-t^{3}-1), \quad 
    p_{6}=-2t^{2}(t^{3}-1), \\  
    p_{7}=4t^{3}(t^{2}-1), \quad p_{8}=\tfrac{1}{20}t^{2}(t^{3}-1),\quad p_{9}=\tfrac{3}{8}t^{3}(t^{3}-1), \quad 
    p_{10}=-\tfrac{1}{6}t^{4}(t^{2}-1),\\ \ \newline
    \end{gathered}$} \\ \hline
 \end{tabular}
\end{center}

\begin{center} \tiny
 \begin{tabular}{|c|c|} \hline
   ${\{Y_1,\dots,Y_8\}}$ & $\mu_{6,\alpha}$ \\ \hline
   
   $\begin{aligned}\ \newline\\  Y_1 &= X_8 \\ Y_2 &= X_1 \\  Y_3 &= X_7 \\  Y_4 &= X_6 \\ 
   Y_5 &= X_5 \\ Y_6 &= X_4 \\ Y_7 &= X_3 \\ Y_8 &= X_2  \\ \ \newline
     \end{aligned}$
   & $\begin{gathered}\ 
       [Y_1,Y_2] = -Y_3, \quad [Y_1,Y_3] = -(2+\alpha)Y_5-Y_6, \\ 
       [Y_1,Y_4] = -(2+\alpha)Y_6-Y_7, \quad [Y_1,Y_5] = -(1+\alpha)Y_7-Y_8, \\ 
       [Y_1,Y_6] = -\alpha Y_8, \quad [Y_2,Y_3] = Y_4, \\
       [Y_2,Y_4] = Y_5, \quad [Y_2,Y_5] = Y_6, \\
       [Y_2,Y_6] = Y_7, \quad [Y_2,Y_7] = Y_8, \\
       [Y_3,Y_4] = -Y_7, \quad [Y_3,Y_5] = -Y_8
     \end{gathered}$ \\ \hline
     
     \tiny{$D=\left(\begin{smallmatrix}
1 &   &   &   &   &   &   \\
  & 3 &   &   &   &   &   \\
  &   & 4 &   &   &   &   \\
  &   &   & 5 &   &   &   \\
  &   &   &   & 6 &   &   \\
  &   &   &   &   & 7 &   \\
  &   &   &   &   &   & 8 \\
\end{smallmatrix}\right)$}  &  \rule[-1.5cm]{0cm}{3.2cm}\tiny{$g_{t}=\left(\begin{smallmatrix} 
t &  &  &  &  &  &  &  \\[0.1cm]
0 & t &  &  &  &  &  &  \\[0.1cm]
0 & p_{1} & t^{3} &  &  &  &  &  \\[0.1cm]
0 & 0 & 0 & t^{4} &  &  &  &  \\[0.1cm]
0 & p_{2} & p_{3} & 0 & t^{5} &  &  &  \\[0.1cm]
p_{4} & p_{5} & p_{6} & p_{7} & 0 & t^{6} &  &  \\[0.1cm]
p_{8} & 0 & p_{9} & p_{10} & p_{11} & 0 & t^{7} &  \\[0.1cm]
0 & 0 & p_{12} & p_{13} & p_{14} & p_{15} & 0 & t^{8}\\
\end{smallmatrix}\right)$}  \\ \hline
   \multicolumn{2}{|l|}{$\begin{gathered}\ \newline\\
    p_{1}=-\tfrac{1}{2}t(t-1),\quad p_{2}=\tfrac{1}{8}t(2+\alpha)(t^{2}-2t+1), \quad p_{3}=-\tfrac{1}{2}t^{3}(2+\alpha)(t-1),
    \\
    p_{4}=\tfrac{1}{8}t(2+\alpha)(1+\alpha)(t^{3}-3t^{2}+3t-1),\quad
    p_{5}=\tfrac{1}{30}t(2t^{3}-5t+3),\quad  p_{6}=-\tfrac{1}{3}t^{3}(t^{2}-1), \\
    p_{7}=-\tfrac{1}{2}t^{4}(2+\alpha)(t-1),\quad
    p_{8}=\tfrac{1}{120}t(t-1)(16t^{3}+20t^{3}\alpha+16t^{2}-8t^{2}\alpha-54t-43t\alpha+30+27\alpha), \\
    p_{9}=\tfrac{1}{8}t^{3}(2+\alpha)(1+\alpha)(t^{2}-2t+1),\quad
    p_{10}=-\tfrac{1}{3}t^{4}(t^{2}-1),\quad p_{11}=-\tfrac{1}{2}t^{5}(1+\alpha)(t-1),\\
    p_{12}=\tfrac{1}{30}t^{3}(t-1)(4t^{2}+5t^{2}\alpha+4t-6-5\alpha),\quad
    p_{13}=\tfrac{1}{8}t^{4}\alpha(2+\alpha)(t^{2}-2t+1),\\
    p_{14}=-\tfrac{1}{3}t^{5}(t^{2}-1),\quad p_{15}=-\tfrac{1}{2}t^{6}\alpha(t-1) \\ \ \newline
    \end{gathered}$} \\ \hline
 \end{tabular}
\end{center}

\begin{center} \tiny
 \begin{tabular}{|c|c|} \hline
   ${\{Y_1,\dots,Y_8\}}$ & $\mu_{9,\alpha}$ \\ \hline
   
   $\begin{aligned}\ \newline\\  Y_1 &= X_8 \\ Y_2 &= X_1 \\  Y_3 &= X_7 \\  Y_4 &= X_6 \\ 
   Y_5 &= X_5 \\ Y_6 &= X_4 \\ Y_7 &= X_3 \\ Y_8 &= X_2  \\ \ \newline
     \end{aligned}$
   & $\begin{gathered}\ 
       [Y_1,Y_2] = -Y_3, \quad [Y_1,Y_3] = -Y_5-Y_6-\alpha Y_7, \\ 
       [Y_1,Y_4] = -Y_6-Y_7-\alpha Y_8, \quad [Y_1,Y_5] = -Y_7, \\ 
       [Y_1,Y_6] = -Y_8, \quad [Y_2,Y_3] = Y_4, \\
       [Y_2,Y_4] = Y_5, \quad [Y_2,Y_5] = Y_6, \\
       [Y_2,Y_6] = Y_7, \quad [Y_2,Y_7] = Y_8, \\
       [Y_3,Y_4] = -Y_8
     \end{gathered}$ \\ \hline
     
     \tiny{$D=\left(\begin{smallmatrix}
1 &   &   &   &   &   &   \\
  & 4 &   &   &   &   &   \\
  &   & 5 &   &   &   &   \\
  &   &   & 6 &   &   &   \\
  &   &   &   & 7 &   &   \\
  &   &   &   &   & 8 &   \\
  &   &   &   &   &   & 9 \\
\end{smallmatrix}\right)$}  &  \rule[-1.5cm]{0cm}{3.2cm}\tiny{$g_{t}=\left(\begin{smallmatrix} 
t &  &  &  &  &  &  &  \\[0.1cm]
0 & t &  &  &  &  &  &  \\[0.1cm]
0 & p_{1} & t^{4} &  &  &  &  &  \\[0.1cm]
p_{2} & 0 & 0 & t^{5} &  &  &  &  \\[0.1cm]
p_{3} & 0 & p_{4} & 0 & t^{6} &  &  &  \\[0.1cm]
p_{5} & p_{6} & p_{7} & p_{8} & 0 & t^{7} &  &  \\[0.1cm]
0 & 0 & p_{9} & p_{10} & p_{11} & 0 & t^{8} &  \\[0.1cm]
0 & 0 & p_{12} & p_{13} & 0 & p_{14} & 0 & t^{9}\\
\end{smallmatrix}\right)$}  \\ \hline
   \multicolumn{2}{|l|}{$\begin{gathered}\ \newline\\
    p_{1}=-\tfrac{1}{3}t(t^{2}-1),\quad p_{2}=-\tfrac{1}{6}t(t-1)(3t^{2}-2t-2), \quad
    p_{3}=\tfrac{1}{5}t(t-1)(11t^{4}-9t^{3}-14t^{2}+5t+5),\\
    p_{4}=-\tfrac{1}{2}t^{4}(t-1),\quad
    p_{5}=-\tfrac{1}{24}t(t-1)(6t^{4}\alpha-3t^{3}+6t^{3}\alpha+3t^{2}+6t^{2}\alpha-8t\alpha-8\alpha), \\
    p_{6}=\tfrac{1}{90}t(t-1)(33t^{4}-22t^{3}-37t^{2}+10t+10),\quad p_{7}=-\tfrac{1}{3}t^{4}(t^{2}-1), \quad
    p_{8}=-\tfrac{1}{2}t^{5}(t-1), \\
    p_{9}=-\tfrac{1}{8}t^{4}(t-1)(2t^{2}\alpha-t+2t\alpha+1+2\alpha),\quad
    p_{10}=-\tfrac{1}{3}t^{5}(t^{2}-1),\quad p_{11}=-\tfrac{1}{2}t^{6}(t-1),\\
    p_{12}=\tfrac{1}{15}t^{4}(t-1)(3t^{2}-2t-2),\quad
    p_{13}=-\tfrac{1}{8}t^{5}(t-1)(2t^{2}\alpha-t+2t\alpha+1+2\alpha),
    p_{14}=-\tfrac{1}{2}t^{7}(t-1) \\ \ \newline
    \end{gathered}$} \\ \hline
 \end{tabular}
\end{center}

\begin{center} \tiny
 \begin{tabular}{|c|c|} \hline
   ${\{Y_1,\dots,Y_8\}}$ & $\mu_{10,\alpha}$ \\ \hline
   
   $\begin{aligned}\ \newline\\  Y_1 &= X_8 \\ Y_2 &= X_1 \\  Y_3 &= X_7 \\  Y_4 &= X_6 \\ 
   Y_5 &= X_5 \\ Y_6 &= X_4 \\ Y_7 &= X_3 \\ Y_8 &= X_2  \\ \ \newline
     \end{aligned}$
   & $\begin{gathered}\ 
       [Y_1,Y_2] = -Y_3, \quad [Y_1,Y_3] = -Y_5-Y_7-\alpha Y_8, \\ 
       [Y_1,Y_4] = -Y_6-Y_8, \quad [Y_1,Y_5] = -Y_7, \\ 
       [Y_1,Y_6] = -Y_8, \quad [Y_2,Y_3] = Y_4, \\
       [Y_2,Y_4] = Y_5, \quad [Y_2,Y_5] = Y_6, \\
       [Y_2,Y_6] = Y_7, \quad [Y_2,Y_7] = Y_8
     \end{gathered}$ \\ \hline
     
     \tiny{$D=\left(\begin{smallmatrix}
1 &   &   &   &   &   &   \\
  & 2 &   &   &   &   &   \\
  &   & 3 &   &   &   &   \\
  &   &   & 4 &   &   &   \\
  &   &   &   & 5 &   &   \\
  &   &   &   &   & 6 &   \\
  &   &   &   &   &   & 7 \\
\end{smallmatrix}\right)$}  &  \rule[-1.5cm]{0cm}{3.2cm}\tiny{$g_{t}=\left(\begin{smallmatrix} 
t &  &  &  &  &  &  &  \\[0.1cm]
0 & t &  &  &  &  &  &  \\[0.1cm]
p_{1} & 0 & t^{2} &  &  &  &  &  \\[0.1cm]
p_{2} & p_{3} & 0 & t^{3} &  &  &  &  \\[0.1cm]
p_{4} & p{5} & p_{6} & 0 & t^{4} &  &  &  \\[0.1cm]
0 & 0 & 0 & p_{7} & 0 & t^{5} &  &  \\[0.1cm]
0 & 0 & p_{8} & 0 & p_{9} & 0 & t^{6} &  \\[0.1cm]
0 & 0 & p_{10} & p_{11} & 0 & p_{12} & 0 & t^{7}\\
\end{smallmatrix}\right)$}  \\ \hline
   \multicolumn{2}{|l|}{$\begin{gathered}\ \newline\\
    p_{1}=-\tfrac{2}{5}t\alpha(t^{4}-1),\quad p_{2}=-\tfrac{1}{8}t(6t^{3}-3t^{2}+10t-13), \quad
    p_{3}=-\tfrac{1}{5}t\alpha(t^{4}-1),\\
    p_{4}=\tfrac{1}{5}t\alpha(t^{4}-1),\quad p_{5}=-\tfrac{1}{8}t(2t^{3}-t^{2}+2t-3), \quad
    p_{6}=-\tfrac{1}{2}t^{2}(t-1),\\
    p_{7}=-\tfrac{1}{2}t^{3}(t-1), \quad p_{8}=-\tfrac{1}{8}t^{2}(2t^{3}-t^{2}+2t-3), \quad
    p_{9}=-\tfrac{1}{2}t^{4}(t-1),\\
    p_{10}=-\tfrac{1}{5}t^{2}\alpha(t^{4}-1),\quad p_{11}=-\tfrac{1}{8}t^{3}(2t^{3}-t^{2}+2t-3),\quad
    p_{12}=-\tfrac{1}{2}t^{5}(t-1) \\ \ \newline
    \end{gathered}$} \\ \hline
 \end{tabular}
\end{center}

\begin{center} \tiny
 \begin{tabular}{|c|c|} \hline
   ${\{Y_1,\dots,Y_8\}}$ & $\mu_{11}$ \\ \hline
   
   $\begin{aligned}\ \newline\\  Y_1 &= X_8 \\ Y_2 &= X_1 \\  Y_3 &= X_7 \\  Y_4 &= X_6 \\ 
   Y_5 &= X_5 \\ Y_6 &= X_4 \\ Y_7 &= X_3 \\ Y_8 &= X_2  \\ \ \newline
     \end{aligned}$
   & $\begin{gathered}\ 
       [Y_1,Y_2] = -Y_3, \quad [Y_1,Y_3] = -Y_5-Y_8, \\ 
       [Y_1,Y_4] = -Y_6, \quad [Y_1,Y_5] = -Y_7, \\ 
       [Y_1,Y_6] = -Y_8, \quad [Y_2,Y_3] = Y_4, \\
       [Y_2,Y_4] = Y_5, \quad [Y_2,Y_5] = Y_6, \\
       [Y_2,Y_6] = Y_7, \quad [Y_2,Y_7] = Y_8
     \end{gathered}$ \\ \hline
     
     \tiny{$D=\left(\begin{smallmatrix}
1 &   &   &   &   &   &   \\
  & 2 &   &   &   &   &   \\
  &   & 3 &   &   &   &   \\
  &   &   & 4 &   &   &   \\
  &   &   &   & 5 &   &   \\
  &   &   &   &   & 6 &   \\
  &   &   &   &   &   & 7 \\
\end{smallmatrix}\right)$}  &  \rule[-1.5cm]{0cm}{3.2cm}\tiny{$g_{t}=\left(\begin{smallmatrix} 
t &  &  &  &  &  &  &  \\[0.1cm]
0 & t &  &  &  &  &  &  \\[0.1cm]
p_{1} & 0 & t^{2} &  &  &  &  &  \\[0.1cm]
p_{2} & p_{3} & 0 & t^{3} &  &  &  &  \\[0.1cm]
0 & p_{4} & p_{5} & 0 & t^{4} &  &  &  \\[0.1cm]
0 & p_{6} & 0 & p_{7} & 0 & t^{5} &  &  \\[0.1cm]
0 & 0 & p_{8} & 0 & p_{9} & 0 & t^{6} &  \\[0.1cm]
0 & 0 & p_{10} & p_{11} & 0 & p_{12} & 0 & t^{7}\\
\end{smallmatrix}\right)$}  \\ \hline
   \multicolumn{2}{|l|}{$\begin{gathered}\ \newline\\
    p_{1}=-\tfrac{8}{5}t(t^{4}-1),\quad p_{2}=\tfrac{1}{8}t(3t^{2}-10t+7), \quad
    p_{3}=-\tfrac{4}{5}t(t^{4}-1),\\
    p_{4}=\tfrac{1}{8}t(t^{2}-2t+1),\quad p_{5}=-\tfrac{1}{2}t^{2}(t-1), \quad
    p_{6}=-\tfrac{1}{5}t(t^{4}-1),\\
    p_{7}=-\tfrac{1}{2}t^{3}(t-1), \quad p_{8}=\tfrac{1}{8}t^{2}(t^{2}-2t+1), \quad
    p_{9}=-\tfrac{1}{2}t^{4}(t-1),\\
    p_{10}=-\tfrac{1}{5}t^{2}(t^{4}-1),\quad p_{11}=\tfrac{1}{8}t^{3}(t^{2}-2t+1),\quad
    p_{12}=-\tfrac{1}{2}t^{5}(t-1) \\ \ \newline
    \end{gathered}$} \\ \hline
 \end{tabular}
\end{center}

\begin{center} \tiny
 \begin{tabular}{|c|c|} \hline
   ${\{Y_1,\dots,Y_8\}}$ & $\mu_{13,\alpha}$ \\ \hline
   
   $\begin{aligned}\ \newline\\  Y_1 &= X_8 \\ Y_2 &= X_1 \\  Y_3 &= X_7 \\  Y_4 &= X_6 \\ 
   Y_5 &= X_5 \\ Y_6 &= X_4 \\ Y_7 &= X_3 \\ Y_8 &= X_2  \\ \ \newline
     \end{aligned}$
   & $\begin{gathered}\ 
       [Y_1,Y_2] = -Y_3, \quad [Y_1,Y_3] = -(1+\alpha)Y_6-Y_7, \\ 
       [Y_1,Y_4] = -(1+\alpha)Y_7-Y_8, \quad [Y_1,Y_5] = -\alpha Y_8, \\ 
       [Y_2,Y_3] = Y_4, \quad [Y_2,Y_4] = Y_5, \\
       [Y_2,Y_5] = Y_6, \quad [Y_2,Y_6] = Y_7, \\
       [Y_2,Y_7] = Y_8, \quad [Y_3,Y_4] = -Y_8
     \end{gathered}$ \\ \hline
     
     \tiny{$D=\left(\begin{smallmatrix}
1 &   &   &   &   &   &   \\
  & 4 &   &   &   &   &   \\
  &   & 5 &   &   &   &   \\
  &   &   & 6 &   &   &   \\
  &   &   &   & 7 &   &   \\
  &   &   &   &   & 8 &   \\
  &   &   &   &   &   & 9 \\
\end{smallmatrix}\right)$}  &  \rule[-1.5cm]{0cm}{3.2cm}\tiny{$g_{t}=\left(\begin{smallmatrix} 
t &  &  &  &  &  &  &  \\[0.1cm]
0 & t &  &  &  &  &  &  \\[0.1cm]
0 & p_{1} & t^{4} &  &  &  &  &  \\[0.1cm]
0 & 0 & 0 & t^{5} &  &  &  &  \\[0.1cm]
p_{2} & 0 & 0 & 0 & t^{6} &  &  &  \\[0.1cm]
p_{3} & 0 & p_{4} & 0 & 0 & t^{7} &  &  \\[0.1cm]
0 & 0 & p_{5} & p_{6} & 0 & 0 & t^{8} &  \\[0.1cm]
0 & 0 & 0 & p_{7} & p_{8} & 0 & 0 & t^{9}\\
\end{smallmatrix}\right)$}  \\ \hline
   \multicolumn{2}{|l|}{$\begin{gathered}\ \newline\\
    p_{1}=-\tfrac{1}{3}t(t^{2}-1),\quad p_{2}=-\tfrac{1}{3}t(1+\alpha)(t^{4}-2t^{2}+1), \quad
    p_{3}=-\tfrac{1}{12}t(3t^{5}-7t^{2}+4),\\
    p_{4}=-\tfrac{1}{3}t^{4}(1+\alpha)(t^{2}-1),\quad p_{5}=-\tfrac{1}{4}t^{4}(t^{3}-1), \quad
    p_{6}=-\tfrac{1}{3}t^{5}(1+\alpha)(t^{2}-1),\\
    p_{7}=-\tfrac{1}{4}t^{5}(t^{3}-1), \quad p_{8}=-\tfrac{1}{3}t^{6}\alpha(t^{2}-1) \\ \ \newline
    \end{gathered}$} \\ \hline
 \end{tabular}
\end{center}

\begin{center} \tiny
 \begin{tabular}{|c|c|} \hline
   ${\{Y_1,\dots,Y_8\}}$ & $\mu_{15}$ \\ \hline
   
   $\begin{aligned}\ \newline\\  Y_1 &= X_8 \\ Y_2 &= X_1 \\  Y_3 &= X_7 \\  Y_4 &= X_6 \\ 
   Y_5 &= X_5 \\ Y_6 &= X_4 \\ Y_7 &= X_3 \\ Y_8 &= X_2  \\ \ \newline
     \end{aligned}$
   & $\begin{gathered}\ 
       [Y_1,Y_2] = -Y_3, \quad [Y_1,Y_3] = -Y_6-Y_7-Y_8, \\ 
       [Y_1,Y_4] = -Y_7-Y_8, \quad [Y_1,Y_5] = -Y_8, \\ 
       [Y_2,Y_3] = Y_4, \quad [Y_2,Y_4] = Y_5, \\
       [Y_2,Y_5] = Y_6, \quad [Y_2,Y_6] = Y_7, \\
       [Y_2,Y_7] = Y_8
     \end{gathered}$ \\ \hline
     
     \tiny{$D=\left(\begin{smallmatrix}
1 &   &   &   &   &   &   \\
  & 2 &   &   &   &   &   \\
  &   & 3 &   &   &   &   \\
  &   &   & 4 &   &   &   \\
  &   &   &   & 5 &   &   \\
  &   &   &   &   & 6 &   \\
  &   &   &   &   &   & 7 \\
\end{smallmatrix}\right)$}  &  \rule[-1.5cm]{0cm}{3.2cm}\tiny{$g_{t}=\left(\begin{smallmatrix} 
t &  &  &  &  &  &  &  \\[0.1cm]
0 & t &  &  &  &  &  &  \\[0.1cm]
p_{1} & 0 & t^{2} &  &  &  &  &  \\[0.1cm]
0 & p_{2} & 0 & t^{3} &  &  &  &  \\[0.1cm]
p_{3} & 0 & 0 & 0 & t^{4} &  &  &  \\[0.1cm]
p_{4} & 0 & p_{5} & 0 & 0 & t^{5} &  &  \\[0.1cm]
0 & 0 & p_{6} & p_{7} & 0 & 0 & t^{6} &  \\[0.1cm]
0 & 0 & p_{8} & p_{9} & p_{10} & 0 & 0 & t^{7}\\
\end{smallmatrix}\right)$}  \\ \hline
   \multicolumn{2}{|l|}{$\begin{gathered}\ \newline\\
    p_{1}=-\tfrac{2}{5}t(t^{4}-1),\quad p_{2}=-\tfrac{1}{5}t(t^{4}-1), \quad
    p_{3}=-\tfrac{1}{3}t(t^{2}-1),\\
    p_{4}=\tfrac{1}{20}t(4t^{4}-5t^{3}+1),\quad p_{5}=-\tfrac{1}{3}t^{2}(t^{2}-1), \quad
    p_{6}=-\tfrac{1}{4}t^{2}(t^{3}-1),\\
    p_{7}=-\tfrac{1}{3}t^{3}(t^{2}-1), \quad p_{8}=-\tfrac{1}{5}t^{2}(t^{4}-1),\\
    p_{9}=-\tfrac{1}{4}t^{3}(t^{3}-1), \quad p_{10}=-\tfrac{1}{3}t^{4}(t^{2}-1) \\ \ \newline
    \end{gathered}$} \\ \hline
 \end{tabular}
\end{center}

\begin{center} \tiny
 \begin{tabular}{|c|c|} \hline
   ${\{Y_1,\dots,Y_8\}}$ & $\mu_{17}$ \\ \hline
   
   $\begin{aligned}\ \newline\\  Y_1 &= X_8 \\ Y_2 &= X_1 \\  Y_3 &= X_7 \\  Y_4 &= X_6 \\ 
   Y_5 &= X_5 \\ Y_6 &= X_4 \\ Y_7 &= X_3 \\ Y_8 &= X_2  \\ \ \newline
     \end{aligned}$
   & $\begin{gathered}\ 
       [Y_1,Y_2] = -Y_3, \quad [Y_1,Y_3] = -Y_7-Y_8, \\ 
       [Y_1,Y_4] = -Y_8, \quad [Y_2,Y_3] = Y_4, \\ 
       [Y_2,Y_4] = Y_5, \quad [Y_2,Y_5] = Y_6, \\
       [Y_2,Y_6] = Y_7, \quad [Y_2,Y_7] = Y_8
     \end{gathered}$ \\ \hline
     
     \tiny{$D=\left(\begin{smallmatrix}
1 &   &   &   &   &   &   \\
  & 2 &   &   &   &   &   \\
  &   & 3 &   &   &   &   \\
  &   &   & 4 &   &   &   \\
  &   &   &   & 5 &   &   \\
  &   &   &   &   & 6 &   \\
  &   &   &   &   &   & 7 \\
\end{smallmatrix}\right)$}  &  \rule[-1.5cm]{0cm}{3.2cm}\tiny{$g_{t}=\left(\begin{smallmatrix} 
t &  &  &  &  &  &  &  \\[0.1cm]
0 & t &  &  &  &  &  &  \\[0.1cm]
p_{1} & 0 & t^{2} &  &  &  &  &  \\[0.1cm]
0 & p_{2} & 0 & t^{3} &  &  &  &  \\[0.1cm]
0 & 0 & 0 & 0 & t^{4} &  &  &  \\[0.1cm]
p_{3} & 0 & 0 & 0 & 0 & t^{5} &  &  \\[0.1cm]
0 & 0 & p_{4} & 0 & 0 & 0 & t^{6} &  \\[0.1cm]
0 & 0 & p_{5} & p_{6} & 0 & 0 & 0 & t^{7}\\
\end{smallmatrix}\right)$}  \\ \hline
   \multicolumn{2}{|l|}{$\begin{gathered}\ \newline\\
    p_{1}=-\tfrac{2}{5}t(t^{4}-1),\quad p_{2}=-\tfrac{1}{5}t(t^{4}-1), \quad
    p_{3}=-\tfrac{1}{4}t(t^{3}-1),\\
    p_{4}=-\tfrac{1}{4}t^{2}(t^{3}-1),\quad p_{5}=-\tfrac{1}{5}t^{2}(t^{4}-1), \quad
    p_{6}=-\tfrac{1}{4}t^{3}(t^{3}-1) \\ \ \newline
    \end{gathered}$} \\ \hline
 \end{tabular}
\end{center}

\begin{center} \tiny
 \begin{tabular}{|c|c|} \hline
   ${\{Y_1,\dots,Y_8\}}$ & $\mu_{8}$ \\ \hline
   
   $\begin{aligned}\ \newline\\  Y_1 &= X_1 \\ Y_2 &= -\tfrac{7}{10}X_5+\tfrac{7}{10}X_6 \\  
   Y_3 &= \tfrac{7}{60}X_{4}+\tfrac{7}{4}X_{5}-\tfrac{7}{5}X_{6}+\tfrac{7}{10}X_{7} \\  
   Y_4 &= -\tfrac{1}{3}X_{4}-\tfrac{2}{3}X_{5}-X_{7}+X_{8} \\ Y_5 &= X_3 \\ 
   Y_6 &= \tfrac{7}{30}X_{4} \\ Y_7 &= -\tfrac{7}{60}X_{4}+\tfrac{7}{60}X_{5} \\ Y_8 &= X_2  \\ \ \newline
     \end{aligned}$
   & $\begin{gathered}\ 
       [Y_1,Y_2] = 6Y_7, \quad [Y_1,Y_3] = Y_2+\tfrac{7}{60}Y_5+\tfrac{9}{2}Y_6-6Y_7, \\ 
       [Y_1,Y_4] = \tfrac{10}{7}Y_2+\tfrac{10}{7}Y_3-\tfrac{1}{3}Y_5-10Y_6-\tfrac{90}{7}Y_7, \\
       [Y_1,Y_5] = Y_8, \quad [Y_1,Y_6] = \tfrac{7}{30}Y_5, \\
       [Y_1,Y_7] = -\tfrac{7}{60}Y_5+\frac{1}{2}Y_6, \quad [Y_2,Y_3] = \tfrac{49}{100}Y_5, \\
       [Y_2,Y_4] = 3Y_6, \quad [Y_3,Y_4] = 6Y_7, \\ 
       [Y_3,Y_7] = -\tfrac{49}{600}Y_8, \quad [Y_4,Y_6] = \tfrac{7}{30}Y_8
     \end{gathered}$ \\ \hline
     
     \tiny{$D=\left(\begin{smallmatrix}
2 &   &   &   &   &   &   \\
  & 3 &   &   &   &   &   \\
  &   & 4 &   &   &   &   \\
  &   &   & 5 &   &   &   \\
  &   &   &   & 6 &   &   \\
  &   &   &   &   & 7 &   \\
  &   &   &   &   &   & 10 \\
\end{smallmatrix}\right)$}  &  \rule[-1.5cm]{0cm}{3.2cm}\tiny{$g_{t}=\left(\begin{smallmatrix} 
\tfrac{1}{t} &  &  &  &  &  &  &  \\[0.1cm]
0 & t^{2} &  & p_{1} &  &  &  &  \\[0.1cm]
0 & 0 & t^{3} &  &  &  &  &  \\[0.1cm]
0 & 0 & 0 & t^{4} &  &  &  &  \\[0.1cm]
0 & p_{2} & p_{3} & p_{4} & t^{5} &  & p_{5} &  \\[0.1cm]
0 & p_{6} & p_{7} & p_{8} & 0 & t^{6} &  &  \\[0.1cm]
0 & p_{9} & p_{10} & p_{11} & 0 & 0 & t^{7} &  \\[0.1cm]
0 & p_{12} & p_{13} & p_{14} & p_{15} & p_{16} & p_{17} & t^{10}\\
\end{smallmatrix}\right)$}  \\ \hline
   \multicolumn{2}{|l|}{$\begin{gathered}\ \newline\\
    p_{1}=-\tfrac{5}{7}t^{3}(t-1),\quad p_{2}=-\tfrac{7}{120}t^{2}(t^{6}-2t^{5}+1), \quad
    p_{3}=\tfrac{7}{120}t^{3}(2t^{5}-6t^{4}+t^{3}+3), \\
    p_{4}=\tfrac{1}{24}t^{3}(t^{7}+t^{6}-14t^{5}+12t^{4}-8t^{3}+7t+1),\quad p_{5}=-\tfrac{7}{120}t^{6}(t-1), \\
    p_{6}=-\tfrac{3}{20}t^{2}(t^{6}-1), \quad p_{7}=-\tfrac{1}{20}t^{3}(2t^{6}-5t^{5}-30t^{4}+33), \\
    p_{8}=-\tfrac{1}{28}t^{3}(3t^{7}-3t^{6}-60t^{5}+140t^{4}-83t+3), \quad p_{9}=\tfrac{6}{5}t^{2}(t^{6}-1), \\ 
    p_{10}=\tfrac{3}{10}t^{3}(t^{6}-5t^{5}+4),\quad p_{11}=\tfrac{1}{7}t^{3}(t^{7}-t^{6}-30t^{5}+24t+6), \\ 
    p_{12}=\tfrac{7}{4800}t^{2}(t^{12}-4t^{11}+14t^{6}-16t^{5}+5),\\
    p_{13}=\tfrac{1}{4800}t^{3}(t^{12}-12t^{11}+68t^{10}+16t^{9}+14t^{6}-168t^{5}+252t^{4}-56t^{3}-115),\\
    p_{14}=\tfrac{1}{20160}t^{3}(t^{13}-5t^{12}-44t^{11}-40t^{10}-224t^{9}-126t^{7}-210t^{6}+1344t^{5}-840t^{4}+1344t^{3}-1095t-105),\\
    p_{15}=\tfrac{1}{5}t^{5}(t^{6}-1), \quad p_{16}=\tfrac{7}{600}t^{6}(t^{6}-1), \quad
    p_{17}=\frac{7}{3600}t^{6}(t^{7}-4t^{6}+9t-6)     \\ \ \newline
    \end{gathered}$} \\ \hline
 \end{tabular}
\end{center}

\end{proof}

\begin{remark}
 In all cases the eigenvalues of $g_t$ on the subspace $\langle Y_2,\dots, Y_8\rangle$ are
 exactly $t^{d_2},t^{d_3},\dots,t^{d_8}$, where $d_2,d_3,\dots,d_8$ are the eigenvalues of the derivation $D$.
 This is evident in all cases where the matrix of $g_t$ is triangular.
 In the remaining case, the last one, it can be checked without difficulty.
\end{remark}

\begin{remark}
 Not all the linear deformations constructed as in \eqref{eqn:linear-deformation} correspond to a degeneration.
 Consider $\mu=\mu_8^{17}$ (see \cite{ABG}), let $\h=\langle X_1,\dots,X_7 \rangle$ and choose 
 \[
     D= \left( \begin{smallmatrix}
         0 \\
         & 1 \\
         & & 1 \\
         & & & 1 \\
         & & & & 1 \\
         & & & & & 1 \\
         & & & & & & 1 
        \end{smallmatrix} \right)  
 \]
 as semisimple derivation of $\h$.
 The linear deformation $\mu_t=\mu+t \mu_D$ does not correspond to a degeneration.
 That is, there is no family $g_t$ of linear isomorphisms satisfying \eqref{eqn:degeneration}.
\end{remark}

\noindent{\bf Acknowlegements.}
This paper is part of the PhD.\ thesis of the first author. 
He thanks CONICET for the Ph.D.\ fellowship awarded that made this possible.


\end{document}